\documentclass[11pt]{amsart}
\usepackage{geometry}                
\usepackage{graphicx}
\usepackage{amssymb}
\usepackage{epstopdf}
\newtheorem{thm}{Theorem}[section]
\newtheorem{lem}[thm]{Lemma}
\newtheorem{cor}[thm]{Corollary}

\newtheorem{clm}[thm]{Claim}
\newtheorem{fact}[thm]{Fact}

\theoremstyle{definition}
\newtheorem{defn}{Definition}[section]

\newtheorem{ques}[defn]{Question}

\theoremstyle{remark}
\newtheorem{rmk}{Remark}[section]

\def\square{\hfill${\vcenter{\vbox{\hrule height.4pt \hbox{\vrule
width.4pt height7pt \kern7pt \vrule width.4pt} \hrule height.4pt}}}$}

\title{Finiteness of mapping class groups and Heegaard distance}
\author{Yanqing Zou}
\date{}                                           
\thanks{ We thank Chao Wang for pointing out Corollary 1.2,  thank Ruifeng Qiu and Ying Zhang for many helpful discussions.This work was partially supported by NSFC No. 12131009 and in part by Science and Technology Commission of Shanghai Municipality (No.22DZ2229014).}

\begin{document}
\begin{abstract}
We prove that the mapping class group of a Heegaard splitting with a distance of at least 3 is finite. However, we have constructed a counterexample with a distance of 2 that disproves this assertion. In addition, the fact that the mapping class group of a Heegaard splitting with a distance of at most 1 is infinite, when combined with our results, provides an answer to the question of the finiteness of mapping class groups as viewed from Heegaard distance.
 
\end{abstract}
\maketitle

\vspace*{0.5cm} {\bf Keywords}: Heegaard Distance, Curve Complex, Mapping class group.\vspace*{0.5cm}


\section{Introduction}
\label{sec1}
Let $S$ be a closed, orientable genus at least 2 surface. Denote its mapping class group by $Mod(S)$. 
Let $H_{1}\cup_{S} H_{2}$ be a Heegaard splitting. Then its mapping class group $Mod(S: H_{1}, H_{2})$ is the subgroup of 
$Mod(S)$, which extends to both $H_{1}$ and $H_{2}$. Usually, given an element $f\in Mod(S)$,  it is hard to know when it  extends to both $H_{1}$ and 
$H_{2}$.  Minsky \cite{Gor} proposed the following question.
\begin{ques}
\label{question1.1}
Is $Mod(S:H_{1}, H_{2})$ finite? finitely generated? finitely presented?
\end{ques}

A Heegaard splitting $H_{1}\cup_{S} H_{2}$ is reducible if there is an essential simple closed curve $c\subset S$ so that it bounds a disk on either side of $S$. Otherwise, it is irreducible. For any irreducible Heegaard splitting, Casson and Gordon \cite{CG} introduce a weakly reducible  Heegaard splitting, i.e., there is s an essential disk $D\subset H_{1}$ (resp. $E\subset H_{2}$) so that $\partial D\cap \partial E=\emptyset$. Otherwise, it is called strongly irreducible.  From the definition, a reducible Heegaard splitting is weakly reducible.  In \cite{Nam}, Namazi proved that the mapping class group of a weakly reducible Heegaard splitting is infinite. Therefore to get finite mapping class group,  we focus on strongly irreducible Heegaard splittings. 

Hempel \cite{Hem} introduce a powerful index-Heegaard distance $d(H_{1}, H_{2})$ for studying Heegaard splitting as follows:
\[d(H_{1}, H_{2})=\min\{n\mid e.s.c.c. a_{i}\subset S, a_{i}\cap a_{i+1}=\emptyset,  ~for ~any~ 0\leq i\leq n-1.\},\]
where $a_{0}$(resp.$a_{n}$) bounds a disk in $H_{1}$ (resp. $H_{2}$). By definition,  a weakly reducible Heegaard splitting corresponds to a distance at most 1 Heegaard splitting  while a strongly irreducible Heegaard splitting has distance at  least 2. In general, high distance Heegaard splittings are generic, see \cite{LM}.

Using Heegaard distance, Namazi \cite{Nam} proved that if $d(H_{1}, H_{2})$ is sufficiently large,  then 
$Mod(S:H_{1}, H_{2})$ is finite; Bestvina-Fujiwara\cite{BF}, Ohshika-Sakuma \cite{OS} strengthened this result  and proved that $Mod(S: H_{1}, H_{2})$ is trivial under  sufficiently large condition; Johnson \cite{Joh} reduced the lower bound and prove that  if $d(H_{1}, H_{2})\geq 4$, then 
$Mod(S: H_{1}, H_{2})$ is embedded into the mapping clas group of M, denoted by $Mod(M)$, as a subgroup. By Thurston's gemetrization theorem proved by Perelman, $M$ is hyperbolic and $Mod(M)$ is finite. Then $Mod(S: H_{1}, H_{2})$ is finite. So we consider distance 2 and 3 Heegaard splittings, i.e., strongly irreducible Heegaard splittings. 

However, Long\cite{Lon} constructed a distance at least 1 Heegaard splitting so that its mapping class group contains a pseudo Anosov map. It means that $Mod(S:H_{1}, H_{2})$ isn't simple as we hoped. In \cite{ZQ}, the auther considered the mapping class group of a locally large distance at least 2 Heegaard splitting and proved that it s mapping class group is finite and so does mapping class group of the corresponding 3-manifold.  Considering the action of mapping class group on  $\mathcal {PML}(S)$, i.e., projective measured lamination space on $S$, we can remove the locally large condition and prove the following result.
\begin{thm}
\label{theorem1.1}
If $d(H_{1}, H_{2})\geq 3$,  then $Mod(S: H_{1}, H_{2})$ is finite.
\end{thm}

\begin{rmk}
By Waldhausen theorem proved by Li\cite{Li01,Li02}, there are only finitely many non-isotopic but genus $g$ Heegaard splittings for $M$. Hence by Theorem \ref{theorem1.1}, $Mod(M)$ is finite. This result gives an alternative way to prove finiteness of $Mod(M)$. 
\end{rmk}
However, among all distance 2 Heegaard splittings, we find an counterexample where  $Mod(S:H_{1}, H_{2})$  contains an infinitely reducible element as follows.
\begin{cor}
\label{cor1.1}
There is a distance 2 Heegaard splitting $H_{1}\cup_{S} H_{2}$ and  two disjoint non-isotopic essential simple closed curves $c_{1}$ and $c_{2}$ in $S$ so that the concatenation of these two Dehn twists $T_{c_{1}}\circ T^{-1}_{c_{2}}$ lies in $Mod(S: H_{1}, H_{2})$. 
\end{cor}
Since $c_{1}$ is not isotopic to $c_{2}$,  $T_{c_{1}}\circ T^{-1}_{c_{2}}$ generates an infinite subgroup of 
$Mod(S:H_{1}, H_{2})$.  We will introduce some lemmas in Section \ref{sec2}, prove Theorem \ref{theorem1.1} in 
Section \ref{sec3} and Corollary \ref{cor1.1} in Section \ref{sec4}.
\section{Some Lemmas}
\label{sec2}
Let $S$ be a closed, orientable surface with genus $g\geq 2$. By Nielson-Thurston's classification,  each element in $Mod(S)$ is either periodic, reducicble or pseudo-anosov.  For each pseudo-anosov element $f\in Mod(S)$, its has a stable lamination $(\mathcal{L}^{+}, \mu_{+})$ and unstable lamination $(\mathcal{L}^{-}, \mu_{-})$ in $\mathcal {PML}(S)$. 

\subsection{Dynamics of $Mod(S)$}
\subsubsection{Pseudo anosov map} Suppose $H$ is a handlebody with boundary surface $S$. Then the collection of essential simple closed curves which bound essential disks in $H$  is embedded into $\mathcal{PML}(S)$, and the closure of this subset, denoted by $\Lambda(H)$, has an empty interior in $\mathcal{PML}(S)$.  It is known that not every element of $Mod(S)$ can extend to the whole $H$. For example, the Dehn twist along a disk-busting curve, i.e., it intersects every disk boundary curve nontrivially. Moreover, almost all elements of $Mod(S)$ can't extend over $H$. For pseudo-anosov elements are generic in $Mod(S)$ and if a pseudo-anosov element $f$ extends over $H$,  then by \cite{BJM}, both of its stable and unstable laminations lie in $\Lambda(H)$, which has empty interior in $\mathcal{PML}(S)$. 

It is known that for stable and unstable lamination of  $f$ and any essential simple closed curve $C\subset S$, $f^{n}(C)$ goes to a geodesic lamination $L^{+}$ that the stable lamination $\mathcal{L}^{+}\subset L^{+}$. So for any $\delta>0$, there is $N$ so that for any $n\geq N$, the Hausdorff distance $d_{H}(f^{n}(C), \mathcal {L}^{+})\leq \delta$. Similarly, we have $d_{H}(f^{-n}(C), \mathcal {L}^{-})\leq \delta$. 
Therefore, for any $f\in Mod(S)$ which extends over $H$, if we choose a disc boundary curve $C$, then for any 
$n$, $f^{n}(C)$ bounds an essential disk in $H$. So there is a sequence of disc boundary curves converging to 
$L^{+}$ under  the Hausdorff metric, which in Definition 2.1\cite{Ack}, $L^{+}$ bounds in handlebody $H$. This result implies that the disk boundary curve in $H$ under $f$ seems to lie in $L^{+}$. Considering the measure $(\mathcal {L}^{+},\mu_{+})$ in $\mathcal {PML}(S)$,  Ackermann\cite{Ack} proved the following lemma. 
\begin{lem} [Lemma 3.3]
\label{lemma2.1}
Let $\delta>0$ be small and let $C$ be a geodesic simple closed curve with $d_{H}(C, \mathcal{L}^{-})<\delta$. Suppose in addition that $C$ and $L^{+}$ both bound in a handlebody $H$. Then for any 
small $\epsilon >0$, there are arcs $a^{+}\subset \mathcal{L}^{+}$, $a^{-}\subset \mathcal {L}^{-}$ such that
$a^{+}\cup a^{-}$ is the boundary of a disk, and
\begin{itemize}
\item $\mu_{+}(a^{-})<\epsilon$;
\item  $m(2\epsilon)-2r\leq \mu_{-}(a^{+})\leq 2M(\epsilon)+2r$, where $r=r(\delta)$.
\end{itemize}
\end{lem}

By Lemma \ref{lemma2.1}, we have the following lemma.
\begin{lem}
\label{lemma2.2}
If $\delta$ small enough, the arc $a^{-}$ in Lemma \ref{lemma2.1} intesects $\mathcal{L}^{+}$  in at most one point. 
\end{lem}
\begin{proof}
Suppose that $a^{-}$ intersects $\mathcal {L}^{+}$ in at least two points $p$ and $q$. Let $\gamma$ be a boundary leaf in $\mathcal {L}^{+}$.  Since $\gamma$ is dense in $\mathcal {L}^{+}$,  $a^{-}$ intersects transversely  $\gamma$ in at least two points.  It means that $\mu_{+}(a^{-})$ is larger than a constant depending on $S$. A contracdition.
\end{proof}

\subsubsection{Infinitely reducible map} Let $f\in Mod(S)$ be an infintely reducible element. Then there are finitely many disjoint and non-isotopic $f$-invariant essential simple closed curves $\{C_{1},...,C_{k}\}$ in $S$. Without loss of generality, we assume that $\{C_{1},..,C_{k}\}$ is the canonical minimal reducible system and each $C_{i}$ bounds no disk in $H$.  Then each component of $S-\cup_{i=1}^{k} C_{i}$ is either a pair of pants or a non-pants compact subsurface $R$ with $\chi(R)\leq -1$. If each component of $\partial H-\{C_{1},...,C_{k}\}$ is a pair of pants, then some power of $f$ is the concatenation of some powers of Dehn twists along these $C_{i}$. 
\begin{lem} [Lemma 2.14, \cite{Oer}]
\label{lemma2.3}
Suppose that each component of $\overline{\partial H-\{C_{1},...,C_{k}\}}$ is a pair of pants.
Then for each $1\leq i\leq k$, there is an essential annulus $A_{i}\subset H$ so that $C_{i}\subset \partial A_{i}$.
\end{lem}
Otherwise, there is at least  one non-pants component $ R\subset \overline{\partial H-\{C_{1},...,C_{k}\}}$.  In this case, either $R$ is compressible or incompressible in $H$. If $R$ is compressible in $H$, then  there is an essential simple closed curve $C\subset R$ disjoint from $\cup_{i=1}^{k} C_{i}$ bounding an essential disk in $H$.  If $R$ is incompressible in $H$, we have the following lemma.

\begin{lem}[Lemma 2.15, \cite{Oer}]
\label{lemma2.4}
If $R$ is incompressible in $H$, then either $H$ is an I-bundle, where one incompressible non-pants component of
$R$ is a horizontal boundary or  there is a
$f$-invariant  essential annulus $A\subset H$ so $\partial A\subset \{C_{1},...,C_{k}\}$.
\end{lem}
\subsection{Subsurface projection and disk complex} 
Let $F\subset S$ be an essential proper subsurface.  For simplicity, we assume that $F$ is neither an annulus nor a pair of pants. For any essential simple closed curve $C\subset S$,  if $C\cap F\neq \emptyset$  up to isotopy,  then either $C\subset F$ or $C\cap F$ consists of essential arcs in $F$. In \cite{MM}, Masur-Minsky defined the subsurface projection $\pi_{F}$ as the concatenation of these two maps $ \eta\circ \phi$, where
\begin{itemize}
\item $\phi(C)=C\cap F$;
\item $\eta(a)$ is the collection of essential simple closed curves of $\partial (\overline{N(a\cup \partial F)})$ in $F$, for any $a\subset C\cap F$.
\end{itemize}
Then they proved that for any two disjoint essential simple closed curve $C_{1}$ and $C_{2}$ in $S$, 
if both of them intersect $F$ nontrivially, then $diam_{\mathcal {C}(F)}(\pi_{F}(C_{1}), \pi_{F}(C_{2}))\leq 2$.

Let $H$ be a handlebody with $\partial H=S$. An essential proper subsurface $F\subset S$ is a compressible hole of $H$ if $F$ is compressible in $H$ and for any essential disk $D\subset H$,  $\pi_{F}(\partial D)\neq \emptyset$. Therefore, for a compressible hole $F$, $\partial F$ bounds no disk in $H$. Moreover, Masur-Schleimer \cite{MS} proved the following lemma. 

\begin{lem}
\label{lemma2.5}
If $F$ is a compressible hole, then for any essential disk $D\subset H$, either $\partial D\subset F$ or 
there is an essential disk $D^{'}\subset H$ and $\partial D^{'}\subset F$ so that $diam_{\mathcal {C}(F)}(\pi_{F}(\partial D), \partial D^{'})\leq 6$. 
\end{lem}

\section{Proof of Theorem \ref{theorem1.1}}
\label{sec3}
\begin{proof}
Let $H_{1}\cup_{S}H_{2}$ be a Heegaard splitting and $Mod(S:H_{1}, H_{2})$ its mapping class group. 
By Nielson-Thurston classification, each element is either periodic, reducible or pseudo anosov.  To prove finiteness of $Mod(S:H_{1}, H_{2})$, it is sufficient to prove that there is no infinitely reducible or 
pseudo anosov element in it. We divide the proof into two subsection 3.1 and 3.2.

\subsection{} Suppose there is a pseudo anosov element $f\in Mod(S:H_{1}, H_{2})$. Then there are stable lamiantion $\mathcal {L}^{+}$ and unstable lamination $\mathcal {L}^{-}$ of $f$ on $S$. Since $f: H_{1}\rightarrow H_{1}$, for any $C$ bounding disk in $H_{1}$, $f^{-n}(C)$ goes to unstable lamination $L^{-}$ of $f$. Moreover, both $C$ and $L^{+}$ bound in $H_{1}$.  Then by Lemma \ref{lemma2.1} and \ref{lemma2.2},  there is an esential simple closed curve $a_{1}^{+}\cup a_{2}^{-}$ bounds a disk in $H_{1}$  so that  $a_{1}^{-}$ intersects 
$L^{+}$ in at most one point. Similarly, there is also $a_{2}^{+}\cup a_{2}^{-}$ bounds a disk in $H_{2}$ with the same property. Therefore 
$a_{1}^{+}$ intersects $a_{2}^{-} $  at most one point. Similar to $a_{1}^{-}$ and $a_{2}^{+}$. 

Since both $a_{1}^{+}$ and $a_{2}^{+}$ are two subarcs of $\mathcal{L}{+}$, $a_{1}^{+}\cup a_{1}^{-}$ intersects $a_{2}^{+}\cup a_{2}^{-}$ at most two points. Moreover, $a_{1}^{+}\cup a_{1}^{-}$ (resp. $a_{2}^{+}\cup a_{2}^{-}$) bounds an essential disk $D_{1}$ (resp.$D_{2}$) in $H_{1}$ (resp. $H_{2}$).  Then $\mid \partial D_{1}\cap \partial D_{2}\mid\leq 2$.  Hence the Heegaard distance $d(H_{1}, H_{2})\leq d(\partial D_{1}, \partial D_{2})\leq 2$. A contradiction.

\subsection{} Suppose that there is an infinitely reducible element $f\in Mod(S:H_{1}, H_{2})$. Let $\{C_{1},...,C_{2}\}$ be the canonical minimal  $f$-invariant curve system. 
\begin{clm}
$d(H_{1}, H_{2})\leq 2$.
\end{clm} 
\begin{proof}
In order to prove this claim, it is sufficient to find two essential disks $D_{1}\subset H_{1}$ and 
$D_{2}\subset H_{2}$ so that $d(\partial D_{1}, \partial D_{2}) \leq 2$.

If some curve of $\{C_{1},..., C_{k}\}$ bounds a disk in $H_{1}$, without loss of generacity, we assume that
$C_{1}$ bounds a disk $D$. Let $D_{1}=D$. Then $\partial D \cap \cup_{i=1}^{k} C_{i}=\emptyset$.
 Otherwise, $\{C_{1},...C_{k}\}$ are incompressible in $H_{1}$. 

Case 3.1. $R=\overline{\partial H_{1}-\{C_{1},...,C_{k}\}}$ consists of pair of pants. Then by Lemma \ref{lemma2.3}, for each $1\leq i\leq k$, there is an essential disk 
$D_{i}$ so that $\partial D_{i}\cap C_{i}=\emptyset$.

Case 3.2. $R$ is compressible in $H_{1}$. Then there is an essential disk 
$D\subset H_{1}$ so that $\partial D\cap \cup_{i=1}^{k} C_{i}=\emptyset$. Let $D_{1}=D$.

Case 3.3. $R$ is incompressible in $H_{1}$. If  $f\mid_{R}$ is  periodic, then some power of $f$ is identity map on $R$. Then some power of $f$ is the concatenation of some powers of Dehn twists along $C_{i}$. So by Lemma \ref{lemma2.3},  for each $1\leq i\leq k$,  there is an essential disk $D_{i}\subset H_{1}$ so that 
$\partial D_{i}\cap C_{i}=\emptyset$. The left case is that the restriction of  $f$ on one component $R_{0}\subset R$ is a pseudo Anosov map. By Lemma \ref{lemma2.4}, $H_{1}$ contains an I-bundle with one horizontal surface $R_{0}$.
\begin{fact}
\label{fact3.1}
$H_{1}$ is not an I-bundle with one horizontal surface $R_{0}$.
\end{fact}
\begin{proof}
Suppose the conclusion  is false.  Then  $H_{1}$ is an product I-bundle over $R_{0}$ or a twisted I-bundle with  the horizontal surface $R_{0}$.  
If  $R_{0}$ is compressible in $H_{2}$, then there is a disk $D_{2}\subset H_{2}$ where $\partial D_{2}\cap \cup_{i=1}^{k}C_{i}=\emptyset$. 
In each case, as $H_{1}$ is not a solid torus,  there is  an essential annulus $A\subset H_{1}$ so that $\partial A \cap \partial D_{2}=\emptyset$. Then $d(H_{1}, H_{2})\leq 2$.
So we also assume that $R_{0}$ is incompressible in $H_{2}$. Since $f\mid_{R_{0}}$ is pseudo anosov,  then by Lemma \ref{lemma2.4}, $H_{2}$ is an I-bundle with a horizontal boundary surface $R_{0}$  or contains an I-bundle with a horizontal boundary surface $R_{0}$. In either of these two cases,  there are two essential annulus $A_{1}\subset H_{1}$ and $A_{2}\subset H_{2}$ so that $\partial A_{1}=\partial A_{2}$. Hence $d(H_{1}, H_{2})\leq 2$. 
\end{proof}
The proof of Fact \ref{fact3.1} implies that $H_{1}$ contains no I-bundle as  $R_{0}$ is an horizontal surface. By the similar argument as above, so does $H_{2}$. Then 
there is disk $D_{i}\subset H_{i}$ disjoint from $C_{1}$ for $i=1,2$.   Hence $d(H_{1}, H_{2})\leq d(\partial D_{1}, \partial D_{2})\leq 2.$ A contracdiction.

\end{proof}

By the discussion in 3.1 and 3.2,  each element $f\in Mod(S:H_{1},H_{2})$ is periodic. Hence $Mod(S:H_{1}, H_{2})$ is a finite group.

\end{proof}
\section{Infinite mapping class group of Heegaard splitting }
\label{sec4}
In this section, we build a distance two Heegaard splitting $H_{1}\cup_{S} H_{2}$ so that $Mod(S:H_{1}, H_{2})$ contains an infinitely reducible element.  

Let $H$ be a genus $g$ hanlebody and $D\subset H$ a nonseparating essential disk. We choose a nonseparating essential simple closed curve $C$ disjoint from and non-isotopic to $\partial D$, which bounds no disk in $H$. Then doing a bandsum between 
$D$ and the annulus $A=C\times I$ produces an essential annulus $A^{'}\subset H$.  It is not hard to see that 
$\partial A^{'}$ is nonseparating in $\partial H$.  Denote $\overline{\partial H-\partial A^{'}}$ by $S^{'}$.
Let $\mathcal {D}(S^{'})=\{\partial D\mid \partial D \cap \partial A^{'}=\emptyset, \text{$D$ is an essential disk in $H$.}\}$

Since $g(S)\geq 2$,  $S^{'}$ supports a pseudo anosov map. Let $h$ be a generic pseudo anosov map on $S^{’}$, i.e.,  of which neither stable nor unstable lamination lies in $\overline{\mathcal {D}(S^{'})}$ in $\mathcal {PML}(S^{'})$.  Then $d_{\mathcal {C}(S^{'})}(\mathcal {D}(S^{'}),  h^{n}(\mathcal {D}(S^{'})))\rightarrow \infty$ as $n\rightarrow \infty$. So there is a number $K$ so that \[d_{\mathcal {C}(S^{'})}(\mathcal {D}(S^{'}),  h^{K}(\mathcal {D}(S^{'})))\geq 15.\]

Since $h$ acts on $S^{'}$, it can be extended into a homeomorphism $h$ of $\partial H$ by defining $h\mid_{ \partial A^{'}}=Id$. Hence $H_{1}\cup_{S} H_{2}=H\cup_{h^{K}}H$ is a Heegaard splitting. Moreover, its distance is 2. 

\begin{clm}
$d(H_{1},H_{2})=2$.
\end{clm}

\begin{proof}
Since for any handlebody $H_{i}$, for $i=1,2$,  there is an essentail disk $D_{i}$ disjoint from $\partial A^{'}$,  $d(H_{1},H_{2})\leq 2$. Suppose that the conclusion is false.  Then $d(H_{1}, H_{2})\leq 1$. Hence 
there are a pair of essential disks $D_{1}\subset H_{1}$ and $D_{2}\subset H_{2}$ so that $\partial D_{1}\cap \partial D_{2}=\emptyset$. 

It is not hard to see that $S^{'}$ is a compressible hole for both of $\mathcal {D}(H_{1})$ and 
$\mathcal {D}(H_{2})$. Then by Lemma \ref{lemma2.5}, for $D_{1}\subset H_{1}$, there is an essential disk $D_{1}^{'}\subset H_{1}$ where $\partial D_{1}^{'}\subset S^{'}$ so that $diam_{\mathcal {C}(S^{'})}(\pi_{S^{'}}(\partial D_{1}),\partial D_{1}^{'})\leq 6$. Similarly,  there is also an essential disk $D_{2}^{'}\subset H_{2}$ so that $diam_{\mathcal {C}(S^{'})}(\pi_{S^{'}}(\partial D_{2}),\partial D_{2}^{'})\leq 6$. Additionally,
since $\partial D_{1}$ is disjoint from $\partial D_{2}$, $diam_{\mathcal {C}(S^{'})}(\pi_{S^{'}}(\partial D_{1}),\pi_{S^{'}}(\partial D_{2})\leq 2$.  In conclusion,  $d_{\mathcal {C}(S^{'})}(\partial D_{1}^{'},\partial D_{2}^{'})\leq 14$. A contradiction.
\end{proof}

Denote $\partial A^{'}=C_{1}\cup C_{2}$. It is not hard to see that $C_{1}$ is not isotopic to  $C_{2}$.
Then the contenation of two Dehn twists $\psi=T_{C_{1}}\circ T_{C_{2}}^{-1}$ has infinitely order.  
To prove that $\psi$ extends over $H_{1}$, it is sufficient to prove that there is a minimal disk system 
$\{D_{1},...,D_{g}\}$ so that $\psi(\{D_{1},...,D_{g}\})$ is also a disk system. Therefore, we only need to prove that $\psi(D_{i})$ is an essential disk in $H_{1}$, for any $1\leq i\leq g$. 

Let $\{D_{1},...,D_{g-1}\}$ be  the collection of $g-1$ pairwise disjoint disks disjoint from $A^{'}$ and $D_{g}$  the disk 
 intersecting  $A^{'}$ in one arc.  Since $C_{1}\cup C_{2}$ are disjoint from $\partial D_{i}$, for $1\leq i\leq g-1$, $\psi(D_{i})=D_{i}$. For $\partial D_{g}$,  $h(\partial D_{g})$ still bounds an essential disk in $H_{1}$. Therefore $\psi(H_{1})=H_{1}$. The argument for $H_{2}$ is similar. So we omit it. 
Hence $\psi=T_{C_{1}}\circ T_{C_{2}}^{-1}\subset Mod(S: H_{1}, H_{2})$ has an infinite order.


Yanqing Zou\\
School of Mathematical Sciences \& Shanghai Key Laboratory of PMMP\\
East China Normal University\\
E-mail: yqzou@math.ecnu.edu.cn


\begin{thebibliography}{99}

\bibitem{Ack}
Ackermann, Robert \emph{ An alternative approach to extending pseudo-Anosovs over compression bodies.} Algebr. Geom. Topol. 15 (2015), no. 4, 2383–2391. 

\bibitem{BF}
Bestvina, Mladen; Fujiwara, Koji \emph{Handlebody subgroups in a mapping class group.} In the tradition of Ahlfors-Bers. VII, 29–50, Contemp. Math., 696, Amer. Math. Soc., Providence, RI, 2017.

\bibitem{BJM}
Biringer, Ian; Johnson, Jesse; Minsky, Yair \emph{Extending pseudo-Anosov maps into compression bodies.} J. Topol. 6 (2013), no. 4, 1019–1042.

\bibitem{CG}
Casson, A. J.; Gordon, C. McA. \emph{Reducing Heegaard splittings.} Topology Appl. 27 (1987), no. 3, 275–283.


\bibitem{Gor}
 Gordon, Cameron McA. Problems.\emph{ Workshop on Heegaard Splittings,} 401–411, Geom. Topol. Monogr., 12, Geom. Topol. Publ., Coventry, 2007.

\bibitem{Hem}
Hempel, John \emph{3-manifolds as viewed from the curve complex.} Topology 40 (2001), no. 3, 631–657. 

\bibitem{Joh}
 Johnson, Jesse \emph{ Mapping class groups of medium distance Heegaard splittings.} Proc. Amer. Math. Soc. 138 (2010), no. 12, 4529–4535. 

\bibitem{Lon}
Long, D. D. \emph{ On pseudo-Anosov maps which extend over two handlebodies.} Proc. Edinburgh Math. Soc. (2) 33 (1990), no. 2, 181–190.

\bibitem{Li01}
Li, Tao \emph{Heegaard surfaces and measured laminations. I. The Waldhausen conjecture.} Invent. Math. 167 (2007), no. 1, 135–177. 

\bibitem{Li02}
 Li, Tao \emph{Heegaard surfaces and measured laminations. II. Non-Haken 3-manifolds.} J. Amer. Math. Soc. 19 (2006), no. 3, 625–657.

\bibitem{LM}
Lustig, Martin; Moriah, Yoav \emph{Are large distance Heegaard splittings generic?} With an appendix by Vaibhav Gadre. J. Reine Angew. Math. 670 (2012), 93–119. 

\bibitem{MM}
Masur, H. A.; Minsky, Y. N. \emph{ Geometry of the complex of curves. II. Hierarchical structure.} Geom. Funct. Anal. 10 (2000), no. 4, 902–974. 

\bibitem{MS}
Masur, Howard; Schleimer, Saul\emph{ The geometry of the disk complex.} J. Amer. Math. Soc. 26 (2013), no. 1, 1–62. 

\bibitem{Nam}
Namazi, Hossein \emph{Big Heegaard distance implies finite mapping class group.} Topology Appl. 154 (2007), no. 16, 2939–2949. 

\bibitem{Oer}
Oertel, Ulrich \emph{Automorphisms of three-dimensional handlebodies.} Topology 41 (2002), no. 2, 363–410.

\bibitem{OS}
Ohshika, Ken'ichi; Sakuma, Makoto \emph{Subgroups of mapping class groups related to Heegaard splittings and bridge decompositions.} Geom. Dedicata 180 (2016), 117–134. 

\bibitem{ZQ}
Zou, Yanqing; Qiu, Ruifeng \emph{Finiteness of mapping class groups: locally large strongly irreducible Heegaard splittings.} Groups Geom. Dyn. 14 (2020), no. 2, 591–605. 

\end{thebibliography}
\end{document}